\documentclass{article}
\usepackage[utf8]{inputenc}
\usepackage{enumerate}
\usepackage{amssymb, bm}
\usepackage{amsthm}
\usepackage{amsmath}
\usepackage{tikz-cd} 
\usepackage{tikz}
\usepackage{bbm}
 \usepackage[all]{xy}
\title{Note on holomorphic Morse inequality tensoring with a coherent sheaf}
\author{Xiaojun WU}
\date{\today}
\newtheorem{mythm}{Theorem}
\newtheorem{mylem}{Lemma}
\newtheorem{myprop}{Proposition}
\newtheorem{myex}{Example}
\newtheorem{mycor}{Corollary}
\newtheorem{mydef}{Definition}

\newtheorem{myconj}{Question}
\begin{document}
\maketitle
  \def\tors{\mathrm{Tors}}
\def\cI{\mathcal{I}}
\def\cJ{\mathcal{J}}
\def\Z{\mathbb{Z}}
\def\Q{\mathbb{Q}}  \def\C{\mathbb{C}}
 \def\R{\mathbb{R}}
 \def\N{\mathbb{N}}
 \def\H{\mathbb{H}}
  \def\P{\mathbb{P}}
 \def\rC{\mathcal{C}}
  \def\nd{\mathrm{nd}}
  \def\d{\partial}
 \def\dbar{{\overline{\partial}}}
\def\dzbar{{\overline{dz}}}
 \def\ii{\mathrm{i}}
  \def\d{\partial}
 \def\dbar{{\overline{\partial}}}
\def\dzbar{{\overline{dz}}}
\def \ddbar {\partial \overline{\partial}}
\def\cN{\mathcal{N}}
\def\cE{\mathcal{E}}  \def\cO{\mathcal{O}}
\def\cF{\mathcal{F}}
\def\cS{\mathcal{S}}
\def\cQ{\mathcal{Q}}
\def\cG{\mathcal{G}}
\def\P{\mathbb{P}}
\def\cI{\mathcal{I}}
\def \loc{\mathrm{loc}}
\def \cC{\mathcal{C}}
\bibliographystyle{plain}
\def \dim{\mathrm{dim}}
\def \Sing{\mathrm{Sing}}
\def \Id{\mathrm{Id}}
\def \rank{\mathrm{rank}}
\def \tr{\mathrm{tr}}
\def \Ric{\mathrm{Ric}}
\def \Vol{\mathrm{Vol}}
\def \RHS{\mathrm{RHS}}
\def \liminf{\mathrm{liminf}}
\def \ker{\mathrm{Ker}}
\def \nn{\mathrm{nn}}
\def \im{\mathrm{Im}}
\def \Alb{\mathrm{Alb}}
\def \Pic{\mathrm{Pic}}
\def \diam{\mathrm{diam}}
\def \Aut{\mathrm{Aut}}
\def \Psh{\mathrm{PSH}}
\def \reg{\mathrm{reg}}
\def \coh{\mathrm{coh}}
\def \Tor{\mathrm{Tor}}
\def\stateparagraph{\vskip7pt plus 2pt minus 1pt\noindent}

\maketitle
\begin{abstract}
In this note, we show various versions of holomorphic Morse inequality tensoring with a coherent sheaf. 
\end{abstract}

In the fundamental work of \cite{Dem85}, the following holomorphic Morse inequalities are proven.
\begin{mythm}
Let $X$ be a compact complex manifold of dimension $n$.
Let $F$ be a holomorphic vector bundle of rank $r$ and $E$ be a holomorphic line bundle with a smooth Hermitian metric on $X$.
Let us denote by $X(q),0 \leq q \leq n$, the open subset of points of $x \in X$ that are of index $q$, i.e. points $x$ at which the $(1,1)$ Chern curvature form $ic_1(E)(x)$ has exactly $q$ negative and $n-q$ positive eigenvalues.
We also put $$X(\leq q)=X(0) \cup X(1) \cup \cdots \cup X(q).$$
For all degrees $q= 0,1, \cdots,n$, the following asymptotic inequalities hold when $k$ tends to $+\infty$.
\begin{enumerate}
\item Morse inequalities:
$$\dim H^q(X,E^k \otimes F)\leq r \frac{k^n}{ n !} \int_{X(q)}(-1)^q(i2 \pi c_1(E))^n+o(k^n).$$
\item Strong Morse inequalities:
$$\sum_{j=0}^q (-1)^{q-j} \dim H^j(X,E^k \otimes F)\leq r \frac{k^n}{ n !} \int_{X(\leq q)}(-1)^q(i2 \pi c_1(E))^n+o(k^n).$$
\item Asymptotic Riemann-Roch formula:
$$\sum_{j=0}^n (-1)^{q-j} \dim H^j(X,E^k \otimes F)= r \frac{k^n}{ n !} \int_{X}(-1)^q(i2 \pi c_1(E))^n+o(k^n).$$
\end{enumerate}
\end{mythm}
It is easy to show that the Morse inequalities and the asymptotic Riemann-Roch formula can be deduced from the strong Morse inequalities.
In this note, we will concentrate on generalizing the strong Morse inequalities.

Note that the factor of fixed vector bundle $F$ contributes only its rank in the asymptotic formula.
Thus it is natural to guess or ask whether the same holds if we change the vector bundle by a coherent sheaf (at least in the torsion-free case).
On the other hand, it is natural to weaken the condition that $X$ is a manifold.

In general, to reduce to the manifold and vector bundle case, one need to control the cohomologies of some torsion sheaves.
To do the induction on dimension, we need to prepare some results on the Grothendieck group on compact complex space.
For more information about the Grothendieck ring of coherent sheaves over complex manifolds, we refer to the paper \cite{Gri} and Section 65.14 \cite{St}.
Let $Z$ be a closed analytic set of some compact complex space $X$.
Let $\coh_Z(X)$ be the set of coherent sheaves on $X$ supported in $Z$.
We have the following equivalent definition.
$$\coh_Z(X)=\{\cF \in \coh(X); \exists N >>0, \cI_Z^N \cF=0\}.$$
It is easy to see that the right-handed side set is contained in $\coh_Z(X)$.
Now we prove the converse inclusion.
Let $i: Z \to X$ be the closed immersion of $Z$ into $X$ with a reduced structural sheaf.
In general, we have a natural morphism for any coherent sheaf $\cF$ on $Z$,
$i^* i_* \cF \to \cF$.
Since $Z$ is closed, by definition, we can easily check that the natural morphism $i^{-1} i_* \cF \to \cF$ is always surjective (which is, in fact, an isomorphism).
Here we denote $i^{-1}$, the inverse image functor in the category of sheaves of abelian groups and $i^*$, the inverse image functor in the category of coherent sheaves.
Since $i^* i_* \cF=i^{-1} i_* \cF \otimes_{i^{-1} \cO_X} \cO_Z$, $i^{-1} \cO_X  \to \cO_Z$ is surjective and taking tensor product is right exact,
the map from $i^{-1} i_* \cF$ to $ i^* i_* \cF$ in the following composition of maps
$$i^{-1} i_* \cF=i^{-1} i_* \cF \otimes_{i^{-1} \cO_X}i^{-1} \cO_X \to i^* i_* \cF \to \cF\otimes_{i^{-1} \cO_X} \cO_Z \to \cF$$
is also surjective.
Since the composition of maps is an isomorphism, $i^* i_* \cF\to \cF$ is also an isomorphism.

Let $\cF$ be a coherent sheaf on $X$ supported in $Z$.
Cover $X$ by finite open sets $U_\alpha$ such that $U_\alpha$ is contained in an open set $\Omega_\alpha$ of $\C^{N_\alpha}$ for some $N_\alpha$.
Let $i_\alpha$ be the inclusion of $U_\alpha$.
$i_{\alpha*} \cF|_{U_\alpha}$ is a coherent sheaf on $\Omega_\alpha$.
Since it is supported in $i_\alpha(Z)$ (with possible shrinking $U_\alpha$ a bit,) there exists $N'_\alpha$ large enough such that
$$\cI_{i_\alpha(Z)}^{N'_\alpha} i_{\alpha*} \cF|_{U_\alpha}=0$$
by Hilbert's Nullstellensatz.
Pull back by $i_\alpha$ gives
$$\cI_{Z}^{N'_\alpha} i^*_\alpha i_{\alpha*} \cF|_{U_\alpha}=0 \to \cI_{Z}^{N'_\alpha} \cF|_{U_\alpha}=0.$$
Since $X$ is compact, $N=\max_\alpha N'_\alpha < \infty$ which shows the inverse inclusion.

Define the Grothendieck group $G_Z(X)$ as $\Z[\coh_Z(X)]$ modulo the following equivalent relation.
$[\cS]+[\cQ]=[\cF] \in G_Z(X)$ if there exists an exact sequence 
$$0 \to \cS \to \cF \to \cQ \to 0$$
for the coherent sheaves $\cS, \cF, \cQ$ supported in $Z$.
Let $G(Z)$ be the Grothendieck group of coherent sheaves on $Z$ (supported in $Z$).
Then we have the following lemma.
\begin{mylem}
We have isomorphism
$$i_{Z*}: G(Z) \simeq G_Z(X).$$
\end{mylem}
\begin{proof}
Notice that since $i_{Z}$ is a closed immersion such that $i_{Z*}$ is exact, the morphism $i_{Z*}$ is well defined between Grothendieck groups.
If $\cF \in \coh_Z(X)$, $\cF$ can be given a finite filtration by the formula $F^i\cF :=\cI^i_Z \cF$.
For any $i$, the associated graded piece $F^i \cF /F^{i+1} \cF$ is in $\coh(Z)$.
Thus we can apply the dévissage theorem to the K-theory group (cf. Theorem 4, Section 5 \cite{Qui}).
\end{proof}
Let $Z_i$ be the irreducible components of $Z$.
We have the following lemma.
\begin{mylem}
The natural morphism
$$\oplus i_{Z_i*}:\oplus_i G_{Z_i}(X) \simeq \oplus_i G(Z_i) \to G_Z(X)$$
is surjective.
\end{mylem}
\begin{proof}
By Lemma 1, it is enough to show that the natural morphism $\oplus_i G(Z_i) \to G(Z)$ is surjective.
We proceed by induction on the number $N$ of the components of $Z$.
Let $Z':=\cup_{i <N} Z_i$.
It is enough to show that $G_{Z'}(Z) \oplus G(Z_N) \to G(Z)$ is surjective.
Let $\cF \in G(Z)$.
Let $\cG$ be the kernel of the natural morphism
$$0 \to \cG \to \cF \to i_{Z_N*}i^*_{Z_N} \cF \to 0.$$
In fact, 
since $Z_N$ is closed, it is easy to see for $z \in Z_N$, $\cF_z \to (i_{Z_N*}i^{-1}_{Z_N} \cF)_z$ is surjective.
On the other hand, 
since $i^{-1}_{Z_N} \cO_X \to \cO_{Z_N}$ is surjective and $i_{Z_N*}$ is right exact,
$$i_{Z_N*}i^{-1}_{Z_N} \cF \to i_{Z_N*}i^{*}_{Z_N} \cF =i_{Z_N*}(i^{-1}_{Z_N} \cF \otimes_{i^{-1}_{Z_N} \cO_X} \cO_{Z_N}) $$
is surjective.
Thus for $z \in Z_N$, $\cF_z \to (i_{Z_N*}i^{*}_{Z_N} \cF)_z$ is surjective.
The surjection outside $Z_N$ is trivial since it is a 0 map.
Over $Z_N \setminus \cup_{i \leq N-1} Z_i$, $i_{Z_N}$ is locally isomorphic, thus $\cF \to i_{Z_N*}i^*_{Z_N} \cF$ is isomorphic.
Thus $\cG$ is supported in $Z'$ and $G_{Z'}(Z) \oplus G(Z_N) \to G(Z)$ is surjective.
(More precisely, by Lemma 65.14.1 \cite{St}, using the adjunction theory, $i_{Z_N*}i^{-1}_{Z_N} \cF=\cF \otimes_{\cO_X} i_{Z_N*} \cO_{Z_N}$. However, we do not need this fact here.)
\end{proof}
\begin{mythm}
Let $\cF$ be a torsion sheaf over a compact irreducible reduced complex space $X$ of dimension $n$.
Assume that Theorem 3 holds for any compact irreducible reduced complex space of dimension strictly less than $n$.
Let $r$ be the dimension of the support of $\cF$.
Let $L$ be a holomorphic line bundle over $X$.
Then we have for any $p$
$$h^p(X, \cF \otimes L^m)=O(m^r).$$
\end{mythm}
\begin{proof}
If we have an exact sequence of torsion sheaf
$$0 \to \cS \to \cF \to \cQ \to 0,$$
the conclusion of the theorem for $\cS$, $\cQ$ implies the conclusion for $\cF$.
In particular, by Lemma 2, without loss of generality, we may assume that the support of $\cF$ is irreducible.
Let $Z$ be the support of $\cF$ and $i: Z \to X$ be the closed immersion.
By Lemma 1, without loss of generality, we can assume that $\cF=i_* \cG$ for some coherent sheaf on $Z$.
Thus it is enough to prove that
for any $p$
$$h^p(Z, \cG \otimes L|_Z^m)=O(m^r).$$
Since the dimension of $Z$ is strictly less than $n$ (see, e.g. Proposition 4.26 Chap. II \cite{agbook}), this is a consequence of holomorphic Morse inequalities provided by Theorem 3.
\end{proof}
With this theorem, it is easy to deduce the holomorphic Morse inequalities by induction on the dimension.
\begin{mythm}
Let $X$ be a compact irreducible reduced complex space of dimension $n$.
Let $\cF$ be a coherent sheaf of rank $r$, and $L$ a holomorphic line bundle with a smooth Hermitian metric on $X$.
Let us denote by $X(q),0 \leq q \leq n$, the open subset of points of $x \in X_{\reg}$ that are of index $q$, i.e. points $x$ at which the $(1,1)$ Chern curvature form $ic_1(L)(x)$ has exactly $q$ negative and $n-q$ positive eigenvalues.
We also put $$X(\leq q)=X(0) \cup X(1) \cup \cdots \cup X(q).$$
For all degrees $q= 0,1, \cdots,n$, the following asymptotic inequalities hold when $m$ tends to $+\infty$.
\begin{enumerate}
\item Morse inequalities:
$$\dim H^q(X,L^m \otimes \cF )\leq r \frac{m^n}{ n !} \int_{X(q)}(-1)^q(i2 \pi c_1(L))^n+o(m^n).$$
\item Strong Morse inequalities:
$$\sum_{j=0}^q (-1)^{q-j} \dim H^j(X,L^m \otimes \cF)\leq r \frac{m^n}{ n !} \int_{X(\leq q)}(-1)^q(i2 \pi c_1(L))^n+o(m^n).$$
\item Asymptotic Riemann-Roch formula:
$$\sum_{j=0}^n (-1)^{q-j} \dim H^j(X,L^m \otimes \cF)= r \frac{m^n}{ n !} \int_{X}(-1)^q(i2 \pi c_1(L))^n+o(m^n).$$
\end{enumerate}
\end{mythm}
\begin{proof}
Since the asymptotic Riemann-Roch formula and Morse inequalities can be deduced from the strong Morse inequalities, we can concentrate on its proof.
The proof is done by induction on the dimension $n$.
If $n=0$, the conclusion is trivial.
Assume that the theorem holds for any compact irreducible complex space of dimension strictly less than $n$.

Notice that we have an exact sequence
$$0 \to \cF_\tors \to \cF \to \cF/ \cF_\tors \to 0.$$
On the other hand, by Theorem 2 and the induction condition, we have that for any $p$
$$\dim H^p(X, \cF_\tors \otimes L^m)=o(m^n).$$
Without loss of generality, we can thus assume that $\cF$ is torsion-free.

Let $\pi: \tilde{X} \to X$ be a modification such that $\pi^* \cF/ \tors$ is locally free (where $\tors$ is the torsion part) by \cite{Ros}.
By projection formula, for any $i \in \N$, we have that
$$R^i \pi_*(\pi^* \cF/ \tors \otimes \pi^* L^m)=R^i \pi_*(\pi^* \cF/ \tors) \otimes  L^m.$$
On the other hand, we have natural morphisms 
$$\cF \to \pi_* \pi^*  \cF \to \pi_*(\pi^* \cF/ \tors) $$
which is generically isomorphic.
Applying Theorem 2 to the kernel and the cokernel of the composition map, we have for any $p$,
$$h^p(X, \cF \otimes L^m)=h^p(X, \pi_*(\pi^* \cF/ \tors \otimes \pi^* L^m))+o(m^n).$$
By Leray spectral sequence, we have that
$$E_2^{p,q}=H^p(X, R^q \pi_*(\pi^* \cF/ \tors \otimes \pi^* L^m))\Rightarrow H^{p+q}(\tilde{X}, \pi^* \cF/ \tors \otimes \pi^* L^m).$$
Recall that if a topological space $X$ is metrizable and locally homeomorphic to a subspace of $\R^{2n}$, then the topological dimension of $X$ is less than $2n$ which is by definition the smallest number such that any sheaf cohomology of strictly larger degree of any sheaf of abelian groups is trivial (cf. Theorem 13.13 Chap IV. \cite{agbook}). 
%Notice that for any coherent sheaf $\cG$ and any $p>2n$, $H^p(X, \cG)=0$ since the topological dimension of $X$ is less than $2n$.
Notice that for any open set $U$ in $X$, $\pi^{-1}(U)$ is metrizable as an open subset of a compact manifold. 
Therefore, for any $x \in X$, $R^q \pi_*(\pi^* \cF/ \tors \otimes \pi^* L^m)_x=\varinjlim_{x \in U} H^q(\pi^{-1}(U),\pi^* \cF/ \tors \otimes \pi^* L^m)=0$ for $q>2n$ since the topological dimension of $\pi^{-1}(U)$ is less than $2n$.
%Even if $X$ is singular, by Urysohn's metrization theorem, $X$ as a compact complex space is metrizable.
%For some reason, since for any $x \in X$, $R^q \pi_*(\pi^* \cF/ \tors \otimes \pi^* L^m)_x=0$ for $q>2n$ since $\pi^{-1}(U)$ is smooth.
Cover $X$ by $N$ finite Stein spaces $U_\alpha$ such that for any $q$, the direct image of $R^q \pi_*(\pi^* \cF/ \tors )$ under the closed embedding of $U_\alpha$ in some euclidean space admits a resolution by vector bundles over $U_\alpha$ of finite length.
By
lemma 4.4 Chap IX \cite{agbook},
we have that for any $q$ and any $\alpha$,
$H^p(U_\alpha,R^q \pi_*(\pi^* \cF/ \tors \otimes \pi^* L^m))=0$
for any $p>0$.
In particular, $H^p(X,R^q \pi_*(\pi^* \cF/ \tors \otimes \pi^* L^m))$ can be calculated as the \u{C}ech cohomology associated to the cover $U_\alpha$.
In particular, $H^p(X,R^q \pi_*(\pi^* \cF/ \tors \otimes \pi^* L^m))=0$ for $p<0$ or $p>N$ or $q<0$ or $q>2n$ for some $N$ independent of $m$ (depending only on the cover).
Thus the Leray spectral sequence will degenerate at most step $\max(N, 2n)$. 
%($2n+1$ if $X$ smooth). 

Notice that by strong Morse inequalities, we have at the limit of the spectral sequence
$$\sum_{j=0}^q \dim H^{j}(\tilde{X}, \pi^* \cF/ \tors \otimes \pi^* L^m)\leq r \frac{m^n}{ n !} \int_{X(\leq q)}(-1)^q(i2 \pi c_1(L))^n+o(m^n).$$
Since for $q>0$, $R^q \pi_*( \pi^* \cF/ \tors)$ is a torsion sheaf,
$\dim E_2^{p,q}=o(m^n)$ by Theorem 2.
Thus we also have that for $q>0$ and any $r \geq 2$
$$\dim E_\infty^{p,q} \leq \dim E_r^{p,q} \leq \dim E_2^{p,q}=o(m^n).$$
On the other hand, since
$$\oplus_{p+q=j} E_{\infty}^{p,q}=H^{j}(\tilde{X}, \pi^* \cF/ \tors \otimes \pi^* L^m),$$
we have that 
$$\sum_{j=0}^q \dim E_\infty^{j,0}\leq r \frac{m^n}{ n !} \int_{X(\leq q)}(-1)^q(i2 \pi c_1(L))^n+o(m^n).$$
By calculation of spectral sequence, for any $r,p,q$, 
$$\dim E_r^{p,q} - \dim E_r^{p+r,q-r+1}-\dim E_r^{p-r,q+r-1} \leq \dim E_{r+1}^{p,q} \leq \dim E_r^{p,q}.$$
Since the spectral sequence degenerates at a finite step, which is independent of $m$, we have in terms of $m$
$$ \dim E_2^{p,0}=\dim E_\infty^{p,0}+o(m^n).$$
This implies that 
$$\sum_{j=0}^q \dim H^j(X, \pi_*(\pi^* \cF / \tors)\otimes L^m)\leq r \frac{m^n}{ n !} \int_{X(\leq q)}(-1)^q(i2 \pi c_1(L))^n+o(m^n)$$
which finishes the proof.
To conclude the asymptotic Riemann-Roch formula, we can use the above arguments to reduce to the locally free case.
In the case when $X$ is smooth, the asymptotic Riemann-Roch formula tensoring with a coherent sheaf is a consequence of
%can be deduced from the asymptotic Riemann-Roch formula tensoring with a fixed vector bundle as follows.
%As shown in \cite{Gra58}, there exists real analytic locally free resolution of the coherent sheaf $\cF$.
the Riemann-Roch-Grothendieck formula (more precisely, Hirzebruch–Riemann–Roch formula) proven by \cite{TT76}.
%for the locally free case implies the lower bound in the asymptotic Riemann-Roch formula.
(In fact, the Riemann-Roch-Grothendieck formula holds for rational Deligne, or rational Bott-Chern cohomology by \cite{Gri} or \cite{Wu20}.)
\end{proof}
As a direct application, we have the following proposition.
\begin{myprop}(Siegel inequality)
Let $X$ be a compact irreducible reduced complex space of dimension $n$.
Let $L$ be a holomorphic line bundle on $X$.
The Kodaira-Iitaka dimension satisfies
$$\limsup_{n \to \infty} \frac{\log h^0(X, L^m) }{\log(m)} \leq n.$$
More generally for any $q \geq 0,$
$$\limsup_{n \to \infty} \frac{\log h^q(X, L^m) }{\log(m)} \leq n.$$
\end{myprop}

An interesting application is a sufficient condition to the Grauert-Riemenschneider conjecture over singular spaces.

\textbf{ Grauert-Riemenschneider conjecture.} If a compact complex manifold $X$ possesses a smooth Hermitian line
bundle, which is semi-positive everywhere and positive on a dense open set, then
$X$ is Moishezon.

Analogous to Demailly's criterion, we have the following result.
\begin{myprop}
Let $X$ be a compact irreducible reduced complex space of dimension $n$.
Let $L$ be a holomorphic line bundle with a smooth Hermitian metric on $X$.
Then 
$X$ is Moishezon (i.e. there exist $n$ algebraically independent meromorphic functions) 
and $L$ is a big line bundle if
$$\int_{X(\leq 1)}(-1)^q(i2 \pi c_1(L))^n >0.$$
\end{myprop}
Note first that since the meromorphic function field is a bimeromorphic invariant.
Siegel-Remmert-Thimm theorem implies that the algebraic dimension is at most $n$.
Since $\pi^* L$ is big by assumption, the algebraic dimension of $\mathcal{M}(X) \simeq \mathcal{M}(\tilde{X})$ is $n$.

Notice that if $X$ is moreover normal, there exists a modification $\pi: \tilde{X} \to X$ such that $\tilde{X}$ is smooth and $\pi_* \cO_{\tilde{X}}=\cO_X$.
Then the fact that $L$ is big is a direct consequence of the projection formula.
Here we do not assume that $X$ is necessarily normal.

Note that Theorem 2 and 3 are new even if $X$ is smooth according to the knowledge of the author.

Bonavero has generalized the holomorphic Morse inequalities to the line bundle with a singular metric with analytic singularities tensoring a fixed vector bundle.
In the following, we discuss the difficulties of our approach when trying to prove the case tensoring a fixed coherent sheaf.
We are able to prove the holomorphic Morse inequalities in some very special cases.
Then we discuss a general approach provided by Demailly.

Let us recall the following definitions.
\begin{mydef} {\rm (Psh~/~quasi-psh functions)}

Let $X$ be a complex manifold (not necessary compact). We say that $\varphi$ is a psh function (resp. a quasi-psh function) on $X$, if $i \d \dbar \varphi \geq 0$,(resp. $i \d \dbar \varphi \geq  \alpha$) in the sense of currents where $\alpha$ is some smooth form on $X$.
%Define the Lelong number of $\varphi$ at $x$ to be
%$$\nu(\varphi,x)= \sup \{\gamma>0 ;\varphi(z) \leq \gamma \log|z-x|+O(1)\; at \;x \}.$$

We say that a quasi-psh function $\varphi$ has analytic singularities, if locally  $\varphi$ is of the form
$$\varphi(z) =c \log(\sum_i |g_i|^2) +C^\infty$$
with $c >0$ and $(g_i)$ some local holomorphic functions.
Here $C^\infty$ means a local smooth function.
Define $\cJ$ as the integral closure of $(g_i)$, which is globally defined on $X$.

If $X$ is a complex space, a quasi-psh function with analytic singularities is the local restriction of a quasi-psh function with analytic singularities via a local embedding of $X$ into some open set of some Euclidean space.
\end{mydef}
\begin{mydef} {\rm (Multiplier ideal sheaf)}.
Let $\varphi$ be a quasi-psh function over a complex manifold $X$. The multiplier ideal sheaves $\cI(\varphi)$ is defined as
$$\cI(\varphi)_x=\{f \in \cO_{X,x}|\exists U_x,\int_{U_x} |f|^2 e^{-2\varphi} < \infty \}$$
where $U_x$ is some open neighbourhood of $x$ in $X$.
\end{mydef}
A basic property of the multiplier ideal sheaf due to \cite{Nad90} is that it is always a coherent ideal sheaf.
We have the following consequence of the Skoda division theorem well known for experts.
\begin{mylem}
Assume that $\varphi=c \log(\sum_{i=1}^k |g_i|^2)+O(1)$ with $c>0$, $g_i \in \cO(\Omega)$ and $\Omega$ open in $\C^N$. 
Denote $\cI:=(g_1, \cdots ,g_k)$.
%Without loss of generality, we can assume that $\cI$ is integrally closed
%since for any $x \in \Omega$,
%$$\overline{\cI}_x=\{f \in \cO_x; \exists C>0,  |f|\leq C \sum |g_i| \}.$$
Then we have for $m$ sufficient large that
$$\cI^{\lfloor \frac{m}{c} \rfloor+1} \subset \cI(m \varphi) \subset \cI^{\lceil \frac{m}{c} -\min\{N,k-1\} \rceil -1}. $$
In particular, for $m$ large enough,
$$V(\cI(m\varphi))=V(g_1, \cdots, g_k)=\Sing(\varphi).$$
\end{mylem}
\begin{proof}
The first inclusion is trivial by definition.
Let $f \in \cO_{\Omega,x}$ such that there exists $x \in U$ an open neighborhood with
$$\int_U \frac{|f|^2}{\sum_i |g_i|^{\frac{2m}{c}}} < \infty.$$
In other words let $f \in \cI(m \varphi)_x$.
Then by Skoda division theorem, for $\frac{m}{c}> \min\{N, k-1\}+1$,
there exist holomorphic functions $h_i$ on $U$ such that $f=\sum g_j h_j$ and for any $j$
$$\int_U \frac{|h_j|^2}{\sum_i |g_i|^{\frac{2m}{c}-2}} < \infty.$$
If $\frac{m}{c}> \min\{N, k-1\}+2$, we can apply again the Skoda division theorem for each $h_j$.
By induction, we can see that there exist holomorphic functions $h_\alpha$ on $U$ such that $f=\sum g^\alpha h_\alpha$ for multi-index $\alpha$ such that 
$$|\alpha| \leq \lceil \frac{m}{c} -\min\{N,k-1\} \rceil -1.$$
This finishes the proof of the inverse inclusion since the index is uniform for any $x \in \Omega$.
\end{proof}
\begin{mycor}
Let $X$ be a compact complex manifold, and $\varphi$ be a quasi psh function on $X$ with analytic singularities.
Then for any $N>0$, there exists  $m>0$  such that
$$\cI(m \varphi)\subset\cI_{\Sing(\varphi)}^{N}  . $$
\end{mycor}
\begin{proof}
Since $V(g_1, \cdots, g_k)=\Sing(\varphi)$ in each local coordinate, 
$(g_1, \cdots, g_k) \subset \cI_{\Sing(\varphi)}$.
Then by the above lemma, the conclusion holds on each local coordinate chart.
Since $X$ is compact, $m$ can be uniform for the finite coordinate charts covering $X$.
\end{proof}
By the above corollary, if the coherent sheaf $\cF$ is torsion and supported in the singular set of the metric with analytic singularities,
for $m >0$ large enough,
$$\cF \otimes \cI(m \varphi)=0.$$
In this case, we have Morse inequalities trivially.
More generally, to study the cohomology group of $\cF \otimes \cI(m \varphi)$ using a resolution of singularities,
however, we need to study the tensor product $\cF \otimes^L \cI(m \varphi)$ in the derived category to calculate the Leray spectral sequence, which is much more complicated as shown in the next example.
\begin{myex}
{\em 
First, let us recall the following algebraic result due to \cite{Kod93}.
Let $R$ be a Noetherian local ring with maximal ideal  $\mathfrak{m}$  and residue field $k$.
Let $\cI$ be an ideal.
Then for any $i$, $\dim_k \Tor_i(\cI^m,k)$ is a polynomial in $m$ for $m$ large enough.
In fact, by \cite[Appendix]{Avr78}, for  all large  $m$ and for any $i$,
$\Tor_i(\cI^m,k) \simeq \Tor_{i+1}(R/\cI^m,k)$.
Then the dimension estimate follows from Theorem 2 of \cite{Kod93}.

Now take $X$, a projective manifold and a point $x \in X$.
Take $R=\cO_{X,x}$.
Let $\varphi$ be a quasi psh function with analytic singularities such that $\cI(m \varphi)=\cI^{c_1 m+c_2}$ for some ideal $\cI$ with $x \in V(\cI)$ and $c_1 \in \N^*, c_2 \in \Z$ for any $m >0$ large enough.
(For example, take $I=\mathfrak{m}_x$ and such $\varphi$ is constucted in Exercise 5.10 \cite{Dem12}.)
Take $\cF=\cO_{X}/\mathfrak{m}_x$ supported at $x$.
Since $\Tor_i(\cI^m,\cF) (\forall i >0)$ is supported at $x$,
%and is also a finitely generated $\cO_{X,x}/\mathfrak{m}_x-$module,
%there exist $N_{m,i}>0$ such that
%$$\Tor_i(\cI^m,\cF) \simeq (\cO_X/\mathfrak{m}_x)^{\oplus N_{m,i}}.$$
we have that
$$h^0(X,\Tor_i(\cI^m,\cF))= \dim_\C \Tor_i(\cI_x^m,\C).$$
%By Lemma 3, $h^0(X, )$
In particular, it can have a polynomial growth in $m$ of degree strictly bigger than the dimension of its support (see e.g. Example 10 \cite{Kod93}).
Therefore, the analogue of Theorem 2 in terms of $ L^{\otimes m} \otimes \Tor_i(\cI(m\varphi),\cO_X/\mathfrak{m}_x)(\forall i>0)$ instead of the tensor product $L^{\otimes m} \otimes \cI(m \varphi) \otimes \cO_X/\mathfrak{m}_x$ does not hold.
}
\end{myex}
Another easy case is the following.
\begin{myprop}
Let $X$ be a compact complex manifold of dimension $n$.
Let $\cF$ be a coherent sheaf of rank $r$, and $L$ a holomorphic line bundle with a singular Hermitian metric on $X$ with quasi-psh potential $\varphi$ with analytic singularities.
Assume that the non-locally-free locus of $\cF$ is disjoint from the pole set of $\varphi$.
Let us denote by $X(q),0 \leq q \leq n$, the open subset of points of $x \in X$ that are of index $q$, i.e. points $x$ at which the $(1,1)$ Chern curvature form $ic_1(L)(x)$ is smooth at $x$ and has exactly $q$ negative and $n-q$ positive eigenvalues.
We also put $$X(\leq q)=X(0) \cup X(1) \cup \cdots \cup X(q).$$
For all degrees $q= 0,1, \cdots,n$, the following asymptotic inequalities hold when $m$ tends to $+\infty$.
\begin{enumerate}
\item Morse inequalities:
$$\dim H^q(X,L^m \otimes \cF \otimes \cI(m \varphi))\leq r \frac{m^n}{ n !} \int_{X(q)}(-1)^q(i2 \pi c_1(L))^n+o(m^n).$$
\item Strong Morse inequalities:
$$\sum_{j=0}^q (-1)^{q-j} \dim H^j(X,L^m \otimes \cF  \otimes \cI(m \varphi))\leq r \frac{m^n}{ n !} \int_{X(\leq q)}(-1)^q(i2 \pi c_1(L))^n+o(m^n).$$
\end{enumerate}
\end{myprop}
\begin{proof}
Since $\cF_{\Tor} \otimes L^m \otimes \cI(m \varphi)=\cF_{\Tor} \otimes L^m$ and $\Tor_i(\cF/\cF_{\Tor}, \cI(m \varphi))=0(\forall i>0, \forall m)$ by assumption and by Theorem 2, we may assume that $\cF$ is torsion free.
By \cite{Ros}, there exists a modification $\pi: \tilde{X} \to X$ such that $\tilde{X}$ is smooth and $\pi^* \cF/\Tor$ is locally free.
Without loss of generality, we may assume $\pi$ is identity over the non-locally-free locus $U$ of $\cF$.
In particular, for any $q \geq 1$, over $U$,
$R^q \pi_*(\pi^* \cF/\Tor \otimes \cI(m \pi^* \varphi))=0=R^q \pi_*(\pi^* \cF/\Tor)$ and $\pi_*(\pi^* \cF/\Tor\otimes K_{\tilde{X}/X} \otimes \cI(m \pi^* \varphi))=\cF \otimes \cI(m \varphi)$.
By Proposition 5.8 \cite{Dem12}, there exist natural morphisms for any $m$,
$$\cF \otimes \cI(m \varphi) \to \pi_*(\pi^* \cF/\Tor \otimes \pi^* \cI(m  \varphi)) \to \pi_*(\pi^* \cF/\Tor \otimes K_{\tilde{X}/X} \otimes \cI(m \pi^* \varphi)).$$
Over any open set $V$ disjoint from the pole set of $\varphi$,
$R^q \pi_*(\pi^* \cF/\Tor \otimes \cI(m \pi^* \varphi))=R^q \pi_*(\pi^* \cF/\Tor)$.
Over $V$, the above natural morphism is independent of $m$ and $\varphi$.
By our assumption, these two kinds of open sets cover $X$.
In particular, the kernel and cokernel of the natural morphism are independent of $m$, which is the kernel and cokernel of $\cF  \to \pi_*(\pi^* \cF/\Tor \otimes K_{\tilde{X}/X})$.
By Theorem 2,
%and vanishing of all higher Tor functors, 
we have estimates of the growth of cohomologies for all appearing torsion sheaves.
The same spectral sequence calculations as in Theorem 3 conclude the proof by changing Demailly's Morse inequalities by Bonavero's applying for $\pi^* \cF/\Tor\otimes K_{\tilde{X}/X} \otimes \cI(m \pi^* \varphi) \otimes \pi^* L^m$.
\end{proof} 

Demailly suggested that
to generalize the Morse inequalities to a line bundle with a singular metric with analytic singularities tensoring with a coherent sheaf, one should consider the following sheaf of $L^2$ sections.
The basic idea is to consider the $L^2$ sections of $K_X \otimes L \otimes \cF$ together instead of considering the $L^2$ sections of $L$ tensoring sections of $\cF$, which will cause difficulties in the spectral sequence calculation.
\begin{mydef}
Let $\cF$ be a coherent sheaf over a compact irreducible Gorenstein normal unibranch (i.e. $\forall x, \cO_{X,x}$ is integral) complex space $X$.
Let $L$ be a line bundle over $X$ with a singular metric with quasi-psh potential $\varphi$.
Locally, there exists a surjective morphism
$$\cO_U^{N_U} \to \cF \to 0.$$
Endow $\cO_U^{N_U}$ with a smooth metric and take the quotient metric on $\cF$, which is well defined on a Zariski open set.
Take a partition of unity of $X$  and glue the quotient metrics to a smooth metric $h$ on $\cF$ well defined on the open set $V$ where $\cF$ is locally free.
Define the sheaf of $L^2$ sections of $K_X \otimes L \otimes \cF$ with respect to $\varphi$ as follows
$$(K_X \otimes L \otimes \cF \otimes \cI(\varphi))^{L^2}_x:= \{s \in (K_X \otimes L \otimes \cF)^{**}_x; \exists x \in U, \int_{U \cap V} |s|_{h}^2 e^{-\varphi} < \infty \}.$$
\end{mydef} 
%By definition, if $\cF$ is torsion-free, $(K_X \otimes L \otimes \cF \otimes \cI(\varphi))^{L^2}$ is a subsheaf of $(K_X \otimes L \otimes \cF)^{**}$.
If $ \cF$ is torsion, then this sheaf is 0.
In general, we have
$$(K_X \otimes L \otimes \cF \otimes \cI(\varphi))^{L^2} =(K_X \otimes L \otimes \cF/ \tors \otimes \cI(\varphi))^{L^2}.$$
It can be checked that the sheaf is independent of the choice of the partition of unity and smooth metrics on $\cO_U^{N_U}$.
If $\cF$ is trivial, 
 $$(K_X \otimes L \otimes \cI(\varphi))^{L^2}= K_X \otimes L  \otimes \cI(\varphi).$$
 Notice that we have the following lemma analogous to the coherence of multiplier ideal sheaf \cite{Nad90}.
\begin{mylem}
Under the assumption of Definition 3, assume that $\cF$ is a torsion-free sheaf.
Let $\pi: \tilde{X} \to X$ be a modification of $X$ such that $\pi^* \cF/ \tors$ is locally free.
%Since $X$ has terminal singularities, and there exists $E>0$ such that
%$K_{\tilde{X}}=\pi^* K_X \otimes \cO(E)$.
Then we have
$$ (K_X \otimes L \otimes \cF \otimes \cI(\varphi))^{L^2}=\pi_* (K_{\tilde{X}} \otimes \pi^* \cF/ \tors  \otimes \pi^* L \otimes \cI(\varphi \circ \pi)).$$
In particular, $(K_X \otimes L \otimes \cF \otimes \cI(\varphi))^{L^2}$ is coherent.
\end{mylem}
\begin{proof}
The proof follows identically as Proposition 5.8 of \cite{Dem12}, which we omit here. 
\end{proof}

\begin{myex}
{\em 
(1) If $\cF$ is locally free over a smooth manifold, then we have
$$(K_X \otimes L \otimes \cF \otimes \cI(\varphi))^{L^2}=K_X \otimes L \otimes \cF \otimes \cI(\varphi)^{}.$$
(2)Let $\cF=\mathfrak{m}_0^k$ the power of the maximal ideal at the origin in $\C^n$.
There exists a surjection
$$E:=\cO_{\C^n}^{\binom{n+k-1}{k}} \to \mathfrak{m}_0^k$$ sending $(u_j)$ to $\sum_j u_j e_j$ with $(e_j)$ a basis of polynomial of degree $k$.
We endow $\cF$ by the quotient metric induced from the trivial metric on $E$. 
We claim that for $k \geq n-1$
$$(K_X \otimes \cF)^{L^2}=K_X \otimes \mathfrak{m}_0^{k-n+1}.$$
On one hand, we have that at $z \neq 0$, the quotient metric of $f$ denoted by $|f(z)|_Q^2$ satisfies 
$$|f(z)|_{Q}^{2} =\inf \sum_i |u_i|^2 \geq \frac{|f(z)|^2}{\sum_j |e_j(z)|^2}  $$
by Schwarz inequality for any $(u_j)\in \C^{\binom{n+k-1}{k}} $ such that $f(z)=\sum_j u_j e_j(z)$.
On the other hand, we have that at the same point
$$|f(z)|_{Q}^{2}  \leq \frac{|f(z)|^2}{\sum_j |e_j(z)|^2}  $$
if we take $u_i=f(z)\overline{e}_i(z)/(\sum_j |e_j(z)|^2)$.
%as constant functions.
In particular, the quotient metric is equivalent to $\frac{1}{|z|^{2k}}$.
%By Skoda division theorem, $(K_X \otimes \cF)^{L^2} \subset K_X \otimes \mathfrak{m}_0^{k-n}.$
Direct calculation concludes (cf. Exercise 5.10 \cite{Dem12}).

(3) More generally, let $\cF$ be an ideal sheaf over a smooth manifold from which we can associate to a quasi psh function $\varphi$ with analytic singularities (taking into account multiplicities).
Then we can check by definition that $(K_X \otimes \cF)^{L^2}=K_X \otimes \cI(\varphi)$.
}
\end{myex}
Sometimes we also can control higher direct images.
\begin{mylem}
With the same notations as in the previous lemma, assume furthermore that $\pi$ is a blow up of smooth center and $ \pi^* \cJ \cdot \cO_{\tilde{X}}= \cO(-D)$.
%with $D$ an SNC divisor having the same support of the exceptional divisor and that $\cI(\varphi \circ \pi)$ admits a metric $h$ such that $ \pi^* \omega_X-\Theta(h) $ is strictly positive for some Hermitian form $\omega_X$ on $X$.
Then for $i>0$, for $k$ large enough,
$$R^i \pi_* (K_{\tilde{X}} \otimes \pi^* \cF/ \tors  \otimes \pi^* L^k \otimes \cI(k \varphi \circ \pi))=0.$$
\end{mylem}
\begin{proof}
The proof follows from the proof of Proposition on Page 38 of \cite{Bon95}.
Without loss of generality, we can assume $U$ is Stein subspace contained in some Stein open set of euclidean space $\Omega$.
We want to apply the usual $L^2$ method to show that for $i>0$,
$$H^i (\pi^{-1}(U), K_{\tilde{X}} \otimes \pi^* \cF/ \tors   \otimes \cI(k \varphi \circ \pi))=0$$
for $k$ large enough.
Let $\psi$ be a strict psh function on $U$. 
Take a smooth metric on $\cO_{\Omega}^{N_U}$. 
%such that the Chern curvature satisfies
%$$\Theta(\cO_{\Omega}^{N_U}) \geq C\omega \otimes \Id$$
%in the sense of Griffiths on $\Omega$ with $\omega$ a smooth K\"ahler metric on ${\Omega}$.
We have a surjective morphism
$$\pi^* \cO_U^{N_U} \to \pi^* \cF/ \tors \to 0$$
since the pullback functor is right exact.
In particular, $\pi^* \cF/ \tors$ is a quotient bundle of $\pi^* \cO_U^{N_U}$ with induced quotient metric $h$.
The quotient metric on a Zariski open set on which $\cF$ is locally free is the same as the quotient metric on $\pi^* \cF/ \tors$ on $\tilde{X} \setminus D$.
%We have that
%$$\Theta(\pi^* \cF/ \tors) \geq C \pi^* \omega \otimes \Id$$
%on $\tilde{X} \setminus E$.
%in the sense of Griffiths.
Notice that the quotient metric $h$ is smooth on $\pi^{-1}(U)$, not only on the preimage of the locally free locus of $\cF$.
%Moreover, the above curvature estimate holds on $\pi^{-1}(U)$.

%Without loss of generality, we can assume $\pi$ is a composition of blows-up of smooth centres by functorial resolution of singularities in the analytic setting proven in Theorem 5.2.2 \cite{Tem12}.
Note that the restriction of $\cO(-D)$ on $D$ is relative ample.
Thus for $C$ large enough (with a possible shrinking of $U$),
$C \pi^* i \d \dbar \psi - \Theta(\cO(D), h_D)$
is strictly positive on $\pi^{-1}(U)$
for some chosen smooth metric $h_D$ on $\cO(D)$.
$\cI(k\pi^* \varphi)=\cO(-\lfloor kc \rfloor D)$ for some $c>0$.
In particular, for $k$ large enough, the metric
$e^{-\lfloor kc \rfloor C \psi} h_D^{-\lfloor kc \rfloor} h $ on $\pi^* \cF/ \tors   \otimes \cI(k \varphi \circ \pi)$
is Nakano positive.
%Write $\cI(\varphi \circ \pi)=\cO(- \sum_i \lambda_i D_i)$ with smooth metric $h_i$ on $\cO(D_i)$.
%In particular, 
%By assumption, $C \pi^* \omega - \sum_i  \lambda_i \Theta(\cO(D_i), h_i)$ is strictly positive for $C$ large enough.
Then the $L^2$ method applies for $C>0, k$ large enough (cf. e.g. Theorem 5.2 of 
\cite{Dem12}). 
Since all smooth metric on $\cO_U^{N_U}$ defines the same sheaf, we can choose a smooth metric such that $C$ is as large as possible. 
\end{proof}
\begin{myconj}(Demailly)

Let $X$ be compact irreducible Gorenstein normal unibranch complex space
and let $L$ be a holomorphic line bundle on $X$ equipped with a singular Hermitian metric.
Let $h=h_\infty e^{-\varphi}$ such that $h_\infty$ is smooth and $\varphi$ is quasi-psh with analytic singularities.
Let $\cF$ be a coherent torsion-free sheaf
on $X$.
Do Morse inequalities hold for $(K_X \otimes L^m \otimes \cF \otimes \cI(m \varphi))^{L^2}$ as $m \to \infty$?
\end{myconj}
If the assumption of Lemma 5 is satisfied,
%for all $L^m \otimes \cI(m \varphi)$ (e.g. when $\pi$ is a single blow-up of smooth centres),
then the Leray spectral sequence degenerates at step two, and we have holomorphic Morse inequalities.
In general, the higher direct images are non-trivial, and it seems difficult to control the asymptotic behaviour.

\paragraph{}
\textbf{Acknowledgement} I thank Jean-Pierre Demailly, my PhD supervisor, for his guidance, patience, and generosity. 
I would like to thank my post-doc mentor Mihai P\u{a}un for many supports.
I would like to thank Junyan Cao for some very useful suggestions on the previous draft of this work.
%In particular, I warmly thank Michel Brion for providing me the arguments of the last section.
%I would also like to express my gratitude to colleagues of Institut Fourier for all the interesting discussions we had. 
This work is supported by DFG Projekt Singuläre hermitianische Metriken für Vektorbündel und Erweiterung kanonischer Abschnitte managed by Mihai P\u{a}un.
  
\end{document}